\documentclass{amsart}

\usepackage{amsfonts}
\usepackage{amssymb}
\usepackage{amsthm}
\usepackage{tikz}

\newtheorem{theorem}{Theorem}[section]
\newtheorem{lemma}[theorem]{Lemma}

\theoremstyle{definition}

\theoremstyle{remarks}

\newcommand{\Z}{\mathbb{Z}}
\newcommand{\gen}[1]{\langle#1\rangle}
\newcommand{\ol}[1]{\overline{#1}}
\newcommand{\normal}{\trianglelefteq}
\newcommand{\lcm}{\mathrm{lcm}}

\begin{document}
\title[Commuting and non-commuting graphs]{Toroidal and projective commuting and non-commuting graphs}
\author[Afkhami, Farrokhi and Khashyarmanesh]{M. Afkhami, M. Farrokhi D. G. and K. Khashyarmanesh}
\thanks{This research was in part supported by grants from IPM (No. 91050011) and (No. 900130063)}
\address{Department of Mathematics, University of Neyshabur, P.O.Box 91136-899, Neyshabur, Iran}
\email{mojgan.afkhami@yahoo.com}
\address{Department of Pure Mathematics, Ferdowsi University of Mashhad, P.O.Box 1159-91775, Mashhad, Iran}
\email{m.farrokhi.d.g@gmail.com}
\address{Department of Pure Mathematics, Ferdowsi University of Mashhad, P.O.Box 1159-91775, Mashhad, Iran}
\address{School of Mathematics, Institute for Research in Fundamental Sciences(IPM), P.O.Box 19395-5746, Tehran, Iran}
\email{khashyar@ipm.ir}

\begin{abstract}
In this paper, all finite groups whose commuting (non-commuting) graphs can be embed on the plane, torus or projective plane are classified.
\end{abstract}
\maketitle
\section{Introduction}
Let $G$ be a non-abelian group. The \textit{commuting graph} associated to $G$ is an undirected graph with vertex set $G\setminus Z(G)$ such that two distinct vertices $x$ and $y$ are adjacent if $xy=yx$. We denote this graph by $\Gamma_G$. Also, the \textit{non-commuting graph} of $G$, which is denoted by $\Gamma'_G$, is an undirected graph with vertex set $G\setminus Z(G)$ such that two distinct vertices $x$ and $y$ are adjacent if $xy\neq yx$. Indeed, $\Gamma'_G$ is the complement of $\Gamma_G$. Commuting graphs as well as non-commuting graphs have many interesting properties, for instance it is known that (non-)commuting graphs characterize non-abelian finite simple groups among all finite groups (see \cite{rms-ajw}).

Recall that a graph is \textit{planar} if it can be drawn in the plane such that its edges intersect only at their end points. A \textit{subdivision} of a graph is any graph that can be obtained from the original graph by replacing edges by paths.
A remarkable characterization of the planar graphs was given by Kuratowski in 1930. Kuratowski's Theorem \cite{kk} states that a graph is planar if and only if it contains no subdivisions of $K_5$ and $K_{3,3}$, where $K_n$ is the \textit{complete graph} with $n$ vertices and $K_{m, n}$ is the \textit{complete bipartite graph} with parts of sizes $m$ and $n$.

It is well-known that a compact surface is homeomorphic to a sphere, a connected sum of $g$ tori, or a connected sum of $k$ projective planes (see \cite [Theorem 5.1]{wm}). We denote $S_0$ for the sphere and $S_g$ ($g\geq1$) for the surface formed by a connected sum of $g$ tori, and $N_k$ for the one formed  by a connected sum of $k$ projective planes. The number $g$ is called the \textit{genus} of the surface $S_g$ and $k$ is called the \textit{crosscap} of $N_k$. When considering the orientability, the surfaces $S_g$ and sphere are among the orientable class of surfaces and the surfaces $N_k$ are among the non-orientable one.

A simple graph which can be embedded in $S_g$ but not in $S_{g-1}$ is called a graph of genus $g$. Similarly, if a simple graph can be embedded in $N_k$ but not in $N_{k-1}$, then we call it a graph of crosscap $k$. The notations $\gamma(\Gamma)$ and $\ol{\gamma}(\Gamma)$ stand for the genus and crosscap of a graph $\Gamma$, respectively. It is easy to see that $\gamma(\Gamma_0)\leq\gamma(\Gamma)$ and $\ol{\gamma}(\Gamma_0)\leq\ol{\gamma}(\Gamma)$, for all subgraphs $\Gamma_0$ of $\Gamma$. Clearly, a graph $\Gamma$ is planar if $\gamma(\Gamma)=0$. A graph $\Gamma$ such that $\gamma(\Gamma)=1$ is called a \textit{toroidal} graph. Also, a graph $\Gamma$ such that $\ol{\gamma}(\Gamma)=1$ is called a \textit{projective} graph.

The aim of this paper is to determine  finite non-abelian groups such that their commuting (or non-commuting) graphs are planar, toroidal or projective.

In this paper, $G$ is a  finite non-abelian group. In the following, we remind some useful theorems that will be used frequently in our proofs. We note that $\lceil x\rceil$ denotes the smallest integer greater than or equal to the given real number $x$.
\begin{theorem}[\cite{gr}]\label{genus}
For positive integers $m$ and $n$, we have
\begin{itemize}
\item[(1)]$\gamma(K_n)=\lceil\frac{1}{12}(n-3)(n-4)\rceil$ if $n\geq3$,
\item[(2)]$\gamma(K_{m,n})=\lceil\frac{1}{4}(m-2)(n-2)\rceil$ if $m,n\geq2$.
\end{itemize}
\end{theorem}
\begin{theorem}[\cite{lwb-fh}]\label{genusbound}
Let $\Gamma$ be a simple graph with $v$ vertices $(v\geq 4)$ and $e$ edges. Then $\gamma(\Gamma)\geq\lceil\frac{1}{6}(e-3v)+1\rceil$.
\end{theorem}
\begin{theorem}[\cite{ab,gr}]\label{crosscap}
For positive integers $m$ and $n$, we have
\begin{itemize}
\item[(1)]$\ol{\gamma}(K_n)=\begin{cases}\lceil\frac{1}{6}(n-3)(n-4)\rceil,&n\geq3\emph{ and }n\neq7,\\
3,&n=7,\end{cases}$
\item[(2)]$\ol{\gamma}(K_{m,n})=\lceil\frac{1}{2}(m-2)(n-2)\rceil$ if $m,n\geq2$.
\end{itemize}
\end{theorem}

A block in a graph is a maximal subgraph with no cut point. The following theorem gives a formula for computing the genus of a graph using its blocks genus.
\begin{theorem}[\cite{jb-fh-yk-jwty}]\label{blocks}
If $\Gamma$ is a graph with blocks $B_1,\ldots,B_n$, then
\[\gamma(\Gamma)=\gamma(B_1)+\cdots+\gamma(B_n).\]
\end{theorem}

Although there is no similar formula for crosscap number of a graph in terms of its blocks crosscap numbers, it is shown in \cite{hg-ph-cw} that $2K_5$ is not projective, the fact that will be used in our proofs.

All over this paper, $\bar{\ }:G\rightarrow G/Z(G)$ denotes the natural homomorphism for a given group $G$, hence $\ol{G}=G/Z(G)$ will denote the image group. Also,  $\omega(G)=\{|x|:x\in G\}$, $\exp(G)=\lcm(\omega(G))$, $Z(G)$ and $S_p(G)$ ($p$ prime) denote the spectrum of $G$, the exponent of $G$, the center of $G$ and a Sylow $p$-subgroup of $G$, respectively. In what follows, $S_n$, $A_n$, $D_{2n}$ and $Q_8$ stand for the symmetric group of degree $n$, alternating group of degree $n$, dihedral group of order $2n$ and the quaternion group of order $8$. Moreover, the union of $n$ disjoint copies of a graph $\Gamma$ will be denoted by $n\Gamma$.
\section{Commuting graphs}
In this section, we will classify all finite non-abelian groups whose commuting graphs can be embedded in the plane, torus or projective plane. We begin with a simple lemma.
\begin{lemma}\label{commutingpgroup}
Let $G$ be a $p$-group of order $p^n$, where $n>1$. Then
\begin{itemize}
\item[(1)]If $p>2$, then $G\setminus\{1\}$ has a commuting subset with $p^2-1\geq8$ elements.
\item[(2)]If $p=2$, $n\geq5$ and $G$ is non-abelian, then $G\setminus Z(G)$ has two disjoint commuting subsets with $6$ elements
\end{itemize}
\end{lemma}
\begin{proof}
(1) Let $x$ be a central element of $G$ of order $p$ and consider the subgroup generated by $\{x,y\}$ for any $y\in G\setminus\gen{x}$.

(2) If $|Z(G)|\geq8$, then consider two distinct cosets of $Z(G)$. Assume $|Z(G)|\leq4$. Let $H$ be a subgroup of $G$ of order $32$ containing $Z(G)$. If $H$ contains an abelian subgroup $K$ of order $16$, then $K\setminus Z(G)$ contains two disjoint commuting subsets with $6$ elements. Hence, we may assume that $H$ does not have abelian subgroups of order $16$. Using the following codes in GAP \cite{tgg}, one can easily see that $|Z(H)|=2$ and consequently $|Z(G)|=2$.
\begin{verbatim}
for i in [1..NrSmallGroups(32)] do
 H:=SmallGroup(32,i);
 if Maximum(List(Filtered(AllSubgroups(H),IsAbelian),Order))<16 then
  Print(Order(Center(H)),"\n");
 fi;
od;
\end{verbatim}
Now, by using following codes, it follows that $H\setminus Z(G)$ contains two disjoint commuting subsets with $6$ elements.
\begin{verbatim}
for i in [1..NrSmallGroups(32)] do
 H:=SmallGroup(32,i);
 L:=Filtered(AllSubgroups(H),IsAbelian);
 counterexample:=true;
 if Maximum(List(L,Order))<16 then
  for A in L do
   for B in L do
    if Order(A)=8 and Order(B)=8 and Order(Intersection(A,B))=2 then
     counterexample:=false;
    fi;
   od;
  od;
 fi;
 if counterexample=true then
  Print(i,"\n");
 fi;
od;
\end{verbatim}
The proof is complete.
\end{proof}
\begin{theorem}\label{planarcommuting}
Let $G$ be a finite non-abelian group. Then $\Gamma_G$ is planar if and only if $G$ is isomorphic to one of the following groups:
\begin{itemize}
\item[(1)]$S_3$, $D_8$, $Q_8$, $A_4$, $D_{10}$, $D_{12}$, $D_8\times\Z_2$, $Q_8\times\Z_2$, $S_4$, $SL(2,3)$, $A_5$,
\item[(2)]$\gen{a,b:a^3=b^4=1,a^b=a^{-1}}\cong\Z_3\rtimes\Z_4$,
\item[(3)]$\gen{a,b:a^4=b^4=1,a^b=a^{-1}}\cong\Z_4\rtimes\Z_4$,
\item[(4)]$\gen{a,b:a^8=b^2=1,a^b=a^{-3}}\cong\Z_8\rtimes\Z_2$,
\item[(5)]$\gen{a,b:a^4=b^2=(ab)^4=[a^2,b]=1}\cong(\Z_4\times\Z_2)\rtimes\Z_2$,
\item[(6)]$\gen{a,b,c:a^2=b^2=c^4=[a,c]=[b,c]=1,[a,b]=c^2}\cong(\Z_4\times\Z_2)\rtimes\Z_2$,
\item[(7)]$\gen{a,b:a^5=b^4=1,a^b=a^3}\cong\Z_5\rtimes\Z_4$.
\end{itemize}
\end{theorem}
\begin{theorem}\label{toroidalprojectivecommuting}
Let $G$ be a finite non-abelian group. Then $\Gamma_G$ is toroidal if and only if $\Gamma_G$ is projective if and only if $G$ is isomorphic to one of the following groups:
\begin{itemize}
\item[(1)]$D_{14}$,
\item[(2)]$D_{16}$,
\item[(3)]$Q_{16}$,
\item[(4)]$QD_{16}$,
\item[(5)]$A_4\times\Z_2$,
\item[(6)]$\gen{a,b:a^7=b^3=1,a^b=a^2}\cong\Z_7\rtimes\Z_3$.
\end{itemize}
\end{theorem}

\textit{Proof of Theorems \ref{planarcommuting} and \ref{toroidalprojectivecommuting}. }
We will show that there are only finitely many groups whose commuting graph have no subgraphs isomorphic to $K_8$ or $2K_5$ and among them, we will cross out those whose commuting graph is not planar, toroidal or projective. We proceed in some steps.

(1) $|Z(G)|\geq8$. Then $xZ(G)$ induces a complete subgraph for each $x\in G\setminus Z(G)$, which is a contradiction. So, we have $|Z(G)|\leq7$.

(2) $|Z(G)|\geq4$. If $\ol{x}\in\ol{G}$ such that $|\ol{x}|>2$, then $xZ(G)\cup x^{-1}Z(G)$ induces a complete subgraph with at least $8$ elements, which is a contradiction. Thus $\ol{G}$ is an elementary abelian $2$-group and hence $G$ is nilpotent. Clearly, $|Z(G)|\neq5,7$. If $|Z(G)|=6$, then $G\cong\Z_3\times H$, where $H$ is an extra special $2$-group. Let $\gen{x}$ be the Sylow $3$-subgroup of $G$. If $A\subseteq H\setminus Z(H)$ is a commuting set, then $\gen{x}\times A$ is a commuting set in $G\setminus Z(G)$. Thus $\Gamma_G$ has a subgraph isomorphic to $K_{3|A|}$. Hence $|A|\leq2$ and this is possible only if $H\cong D_8$ or $Q_8$. Therefore $G\cong\Z_3\times D_8$ or $\Z_3\times Q_8$, which is impossible for $\Gamma_G\cong3K_6$. If $|Z(G)|=4$, then $G$ is a $2$-group and, by Lemma \ref{commutingpgroup}, it follows that $|G|=16$.

(3) $|Z(G)|=3$. If $\ol{x}\in\ol{G}$ is an element of order $>3$, then $xZ(G)\cup x^2Z(G)\cup x^3Z(G)$ induces a complete subgraph isomorphic to $K_9$, which is impossible. Thus $\omega(\ol{G})\subseteq\{1,2,3\}$. With a same argument one can show that $C_G(x)=\gen{Z(G),x}$ for all $x\in G\setminus Z(G)$. Now, we have three cases. If $\ol{G}$ is a $2$-group, then $G$ is abelian, which is a contradiction. Also, if $\ol{G}$ is a $3$-group and $x,y\in G$ are such that $xy\neq yx$, then $xZ(G)\cup x^{-1}Z(G)\cup yZ(G)\cup y^{-1}Z(G)$ induces a subgraph isomorphic to $2K_6$, which is a contardiction. Therefore, $\ol{G}$ is neither a $2$-group nor a $3$-group. Then, by \cite{edb}, either $\ol{G}\cong(\Z_2\times\Z_2)^m\rtimes\Z_3$ or $\ol{G}\cong\Z_3^m\rtimes\Z_2$. If $\ol{G}\cong(\Z_2\times\Z_2)^m\rtimes\Z_3$, then $Z(G)S_2(G)\setminus Z(G)$ induces a complete subgraph with at least $9$ elements, which is a contradiction. Thus $\ol{G}\cong\Z_3^m\rtimes\Z_2$. By previous arguments, $S_3(G)$ must be abelian, which implies that $|S_3(G)\setminus Z(G)|\leq7$. Hence, $|S_3(G)|=9$ and so $|G|=18$. 

(4) $|Z(G)|=2$. If there is an element $\ol{x}\in \ol{G}$ with $|\ol{x}|\geq 5$, then  $xZ(G)\cup x^2Z(G)\cup x^3Z(G)\cup x^4Z(G)$ induces a subgraph isomorphic to $K_8$, which is impossible. Therefore, $\omega(\ol{G})\subseteq \{1,2,3,4\}$. Since $Z(G)\subseteq S_2(G)$, by Lemma \ref{commutingpgroup}, $|G|\big|2^4\cdot3$.

(5) $|Z(G)|=1$. Clearly, $\omega(G)\subseteq\{1,2,3,4,5,6,7,8\}$. By Lemma \ref{commutingpgroup}, $|G|\big|2^4\cdot3\cdot5\cdot7$. Also, if $7\in\omega(G)$, then $S_7(G)\normal G$, which implies that $|G|\big|48$.

Now, the result follows by a simple computation with GAP \cite{tgg}. The converse is straightforward.$\hfill\Box$
\section{Non-commuting graphs}
In this section, we shall determine all finite non-abelian groups whose non-commuting graphs can be embedded in the plane, torus or projective plane. The following theorem of Abdollahi, Akbari and Maimani gives all planar non-commuting graphs.
\begin{theorem}[\cite{aa-sa-hrm}]
Let $G$ be a finite non-abelian group. Then $\Gamma'_G$ is planar if and only if $G$ is isomorphic to one of the groups $S_3$, $D_8$ or $Q_8$.
\end{theorem}
\begin{theorem}
There is no toroidal non-commuting graph.
\end{theorem}
\begin{proof}
Assume on a contrary that $G$ is a finite group with toroidal non-commuting graph. Let $k(G)$ be the number of conjugacy classes of $G$. Since $|V(\Gamma'_G)|=|G|-|Z(G)|$ and 
\begin{align*}
2|E(\Gamma'_G)|&=|G|^2-|\{(x,y)\in G\times G:xy=yx\}|\\
&=|G|^2-|G|k(G),
\end{align*}
by Theorem \ref{genusbound}, it follows that $|G|(|G|-k(G)-6)+6|Z(G)|\leq 0$. Hence $k(G)\geq|G|-5$. On the other hand, $k(G)/|G|\leq5/8$ (see \cite{whg}), from which it follows that $|G|\leq13$. A simple verification shows that $S_3$, $D_8$ and $Q_8$ are the only groups with these properties each of which has a planar non-commuting graph, a contradiction.
\end{proof}
\begin{theorem}
There is no projective non-commuting graph.
\end{theorem}
\begin{proof}
Suppose on the contrary that $G$ is a finite group with projective non-commuting graph. If $x,y\in G$ are such that $xy\neq yx$, then the subgraph induced by $xZ(G)\cup yZ(G)$ is isomorphic to $K_{|Z(G)|,|Z(G)|}$, which implies that $|Z(G)|\leq3$. On the other hand, if $x\in G\setminus Z(G)$, $y\in G\setminus C_G(x)$ and $X$ is the set of all generators of $\gen{x}$, then the subgraph induced by $X\cup\gen{x}y$ is isomorphic to $K_{\varphi(|x|),|x|}$, where $\varphi$ is the Euler totient function, from which it follows that $|x|\leq4$ or $|x|=6$. If $|x|=6$ then there exists a suitable power $x^*$ of $x$ such that $x^*\in G\setminus C_G(y)$ and the subgraph induced by $\{x,x^{-1},x^*\}\cup\gen{x}y$ is isomorphic to $K_{3,6}$, which is a contradiction. Therefore, $\omega(G)\subseteq\{1,2,3,4\}$. On the other hand, if $x\in G$ such that $|\ol{x}|=4$, then the subgraph induced by $\{x,x^{-1},x^2\}\cup(G\setminus C_G(x^2))$ has a subgraph isomorphic to $K_{3,|G\setminus C_G(x^2)|}$, which implies that $|G\setminus C_G(x^2)|\leq4$. Hence, $|G|=8$ and consequently $\exp(\ol{G})=2$, which is a contradiction. Therefore, $\omega(\ol{G})\subseteq\{1,2,3\}$. Since $G$ has no elements of order $6$, it follows that $G$ is a $3$-group, $G$ is a $2$-group or $Z(G)=1$. Thus, we have the following cases:

Case 1. $G$ is a $3$-group. If $x\in G\setminus Z(G)$ and $y\in G\setminus C_G(x)$, then the subgraph induced by $xZ(G) \cup x^{-1}Z(G)\cup yZ(G)$ is isomorphic to $K_{3,6}$, which is a contradiction.

Case 2. $G$ is a $2$-group. Then $|Z(G)|=\exp(\ol{G})=2$, which implies that $G$ is an extra special $2$-group. So, $G=G_1\circ\dots\circ G_n$ is the central product of $G_1,\ldots,G_n$, where $G_i\cong D_8$ or $Q_8$, for $i=1,\dots,n$. Let $x,y \in G_1$ with $xy\neq yx$. If $n>1$, then the subgraph induced by $xG_2 \cup yG_2$ is isomorphic to $K_{8,8}$, which is impossible. Thus $n=1$ and subsequently $G\cong D_8$ or $Q_8$, a contradiction.

Case 3. $|Z(G)|=1$. Let $P=S_2(G)$, $Q=S_3(G)$ and $x,y\in G$ be elements of orders $2$ and $3$, respectively. By Case 2, either $P$ is abelian, or $P\cong D_8$ or $Q_8$. If $P$ is abelian, then the subgraph induced by $(P\setminus\{1\})\cup Py$ is isomorphic to $K_{|P|-1,|P|}$, which implies that $|P|\leq4$. Hence, $|P|\big|8$ in all cases. On the other hand, $Q$ is abelian, which implies that the subgraph induced by $(Q\setminus\{1\})\cup Qx$ is isomorphic to $K_{|Q|-1,|Q|}$. So, we have $|Q|=3$. Therefore  $|G|\big|24$. The only groups with these properties are $S_3$, $A_4$ and $S_4$ each of which has a non-projective non-commuting graph. The proof is complete.
\end{proof}

\noindent{\bf Acknowledgments.}
The authors are deeply grateful to the referee for careful reading of the manuscript and helpful suggestions. \\[0.1cm]


\begin{thebibliography}{0}
\bibitem{aa-sa-hrm}A. Abdollahi, S. Akbari and H.R. Maimani, Non-commuting graph of a group, \textit{J. Algebra} \textbf{298} (2006), 468--492.

\bibitem{jb-fh-yk-jwty}J. Battle, F. Harary, Y. Kodama and J. W. T. Youngs, Additivity of the genus of a graph, \textit{Bull. Amer. Math. Soc.} \textbf{68} (1962), 656--568.

\bibitem{lwb-fh}L.W. Beineke and F. Harary, Inequalities involving the genus of a graph and its thickness, \textit{Prcoc. Glasgow Math. Assoc.} \textbf{7} (1965), 19--21.

\bibitem{edb}E.D. Bolker, Groups whose elements are of order two or three, \textit{Amer. Math. Monthly} \textbf{79}(9) (1972), 1007--1010.

\bibitem{ab}A. Bouchet, Orientable and nonorientable genus of the complete bipartite graph, \textit{J. Combin. Theory Ser. B} \textbf{24} (1978), 24--33.

\bibitem{tgg}The GAP Group, \textit{GAP-Groups, Algorithms and Programming, Version 4.6.4, 2013} (http://www.gap-system.org/).

\bibitem{hg-ph-cw}H.H. Glover, J.P. Huneke and C.S. Wang, $103$ Graphs that are irreducible for the projective plane, \textit{J. Combin. Theory Ser. B} \textbf{27} (1979), 332--370.

\bibitem{whg}W.H. Gustafson, What is the probability that two group elements commute?, \textit{Amer. Math. Monthly} \textbf{80} (1973), 1031--1304.

\bibitem{kk}K. Kuratowski, Sur le probl\`{e}me des courbes gauches en topologie, \textit{Fund. Math.} \textbf{15} (1930), 271--283.

\bibitem{wm}W. Massey, \textit{Algebraic Topology: An Introduction}, Harcourt, Brace $\&$ World, Inc., New York, 1967.

\bibitem{gr}G. Ringel, \textit{Map Color Theorem}, Springer-Verlag, New York, Heidelberg, 1974.

\bibitem{rms-ajw}R. M. Solomon and A. J. Woldar, Simple groups are characterized by their non-commuting graphs, \textit{J. Group Theory} \textbf{16}(6) (2013), 793--824.
\end{thebibliography}
\end{document}